\documentclass[
final, nomarks
]{dmtcs-episciences}

\usepackage[utf8]{inputenc}
\usepackage{subfigure}

\usepackage[round]{natbib}
   
\usepackage{amsmath,amsfonts,amssymb,amsopn,amsthm}

\newtheorem{theorem}{Theorem}
\newtheorem{lemma}[theorem]{Lemma}
\newtheorem{claim}[theorem]{Claim}
\newtheorem{definition}{Definition}

\author[Derman Keskinkılı\c{c} et al.]{Derman Keskinkılı\c{c}\affiliationmark{1}
  \and Lale \"Ozkahya\affiliationmark{1}}
\title[Coloring Grids Avoiding Bicolored Paths]{Coloring Grids Avoiding Bicolored Paths}
\affiliation{
  Hacettepe University, Department of Computer Engineering, Ankara, Turkey.}

\begin{document}
% This is only used if you are compiling for a volume before vol 25
% \publicationdetails{VOL}{2015}{ISS}{NUM}{SUBM}
% This is the new form of collecting the data, starting with vol 25
%\publicationdata
%{vol. 25:3 special issue for main purpose}
%{2022}
%{1}
%{10.46298/dmtcs.10472}
\publicationdata
{vol. 28:3}
{2026}
{6}
{10.46298/dmtcs.15759}
{2025-05-28; 2025-05-28; 2025-12-08; 2026-03-24; 2026-07-17}{2026-07-27}
%{1998-10-14; 1998-10-14; 2002-07-19; 2014-02-05; 2015-09-09; 2022-12-25}
%{2022-12-3}
%{2022-12-3; None}
%{2023-1-1}
\maketitle
\begin{abstract}
\vspace*{2ex} % problem with dmtcs template
The star chromatic number of a graph 
is the minimum number of colors in a proper 
vertex coloring forbidding any $P_4$ with two colors (bicolored). This problem was introduced by Gr\"unbaum (1973) together with the acyclic coloring of graphs, where bicolored cycles are avoided.   
In this paper, we study a generalization of this problem, by considering proper vertex coloring on graphs forbidding bicolored  paths of a fixed length, which was initially discussed by Alon, McDiarmid, and Reed (1991). Here, we study this problem on products of two paths. 
We show that at least 4 colors are needed to properly color the  product of paths, $P_m\square P_n$, avoiding a bicolored $P_k,$ unless $n<k-2$ or $m<k-2.$ With this result, the above question is settled for all $k$ on 2-dimensional grids. 
\end{abstract}

\section{Introduction} \label{sec:introduction}

For a graph $G$, $\chi(G)$ denotes the {\it chromatic number} of $G$, that is, the minimum number of colors needed to color $V(G)$ without a monochromatic edge. 
The {\it star coloring} problem on a graph $G$ 
asks to find the minimum number of colors in a proper 
coloring forbidding a bicolored (2-colored) $P_4,$ called the star-chromatic 
number $\chi_s(G)$. This problem was introduced by Gr\"unbaum \cite{grunbaum1973acyclic}, who proved that a graph with maximum degree 3 has an acyclic coloring with 4 colors. 
An {\it acyclic coloring} of a graph $G$ was introduced in~\cite{grunbaum1973acyclic}, as the proper coloring of $G$ 
not having any bicolored cycle. The minimum number of colors needed in an acyclic coloring of a graph $G$ is called the {\it acyclic chromatic number} of $G$, denoted by $a(G)$, where $a(G)\leq \chi_s(G).$ Both the star coloring and 
acyclic coloring problems were shown to be NP-complete by Albertson et al.~\cite{albertson2004coloring} and Kostochka~\cite{kostochka1978upper}, respectively. 

The star and acyclic coloring problems are also related to the proper coloring of $G^2$, the square of a graph $G$, 
where $V(G^2)=V(G)$ and any two vertices are adjacent in $G^2$ 
 if and only if they are at distance at most two in G. 
In that sense, every proper coloring of $G^2$ is a star and acyclic coloring of $G.$ 
We study a variation of these problems avoiding bicolored paths. We call a proper vertex 
coloring of a graph $G$ without a bicolored copy of $P_k$ 
a {\it $P_k$-coloring} of $G$ for $k\geq 3$.  
The minimum number of colors needed for a $P_k$-coloring 
of $G$ is called {\it $P_k$-chromatic number} of $G,$ 
denoted by $s_k(G).$ We note that, for $k\geq 4$, 
\begin{displaymath}
\chi(G) \leq s_k(G) \leq s_4(G) = \chi_s(G)\leq \chi(G^2)=s_3(G).   
\end{displaymath}

Another relevant problem to $P_k$-coloring is nonrepetitive graph coloring. 
For a path $P=(v_1,\dots,v_{2k}),$ we say $P$ is {\it repetitively} colored if the color of $v_i$ and $v_{k+i}$ are the same for each $i$, $1\leq i\leq k.$ A coloring of a graph $G$, where no path with even vertices is repetitively colored, is called a {\it nonrepetitive coloring} of $G.$ The {\it (path-)nonrepetitive chromatic number}, $\pi(G)$, is the minimum
integer $k$ such that $G$ admits a nonrepetitive $k$-coloring.
Since the seminal result by Alon et al. \cite{alon2002nonrepetitive}, nonrepetitive colorings of graphs have been widely studied. Wood \cite{wood2021nonrepetitive} provides an extensive survey on this coloring. 
The $P_k$ coloring problem, $k\geq 4$, is closely related to this problem as  
\begin{displaymath}
s_k(G) \leq  \chi_s(G)\leq \pi(G).    
\end{displaymath}
The $P_k$-coloring problem was posed by Alon, McDiarmid, and Reed \cite{alon1991acyclic} who showed that for any graph $G$ with maximum degree $d,$ $a(G)= O(d^\frac{4}{3})$, and claimed  that a similar upper bound can be shown for $P_k$-chromatic number.  
This coloring parameter is studied for paths of even order in \cite{esperet2013acyclic} and 
later for all paths in \cite{hou2020coloring, kirticsouglu2024coloring}. 

In this paper, we consider the $P_k$-coloring of 2-dimensional grid graphs, $P_m\square P_n$. Graph products have been studied widely for the problems listed above. In \cite{fertin2003acyclic}, Fertin, Godard and Raspaud provide both lower and upper bound for the acyclic chromatic number of $d$-dimensional grid, that is product of $d$ paths. In \cite{jamison2008acycliccycles}, Jamison and Matthews find the acyclic chromatic number of $P_m\square C_n$ (cylinder) and $C_m\square C_n$ (torus) as at most 4 and 5, respectively. Borodin \cite{borodin1979acyclic} showed that every planar graph has an acyclic 5-coloring. Alon, Mohar, and Sanders \cite{alon1996acyclic} showed that each projective plane graph has an acyclic 7-coloring. Acyclic coloring of products of graphs are extensively studied also in \cite{jamison2006acyclic, jamison2008acyclic}. 
In all these results, we observe that the difference between the chromatic number and acyclic chromatic number is a constant value, independent from the length of the cycles that are avoided to be bicolored. Similarly, 
the most recent result on nonrepetitive coloring of grids is studied in \cite{tao2026nonrepetitive}, showing that 
$5\leq \pi(P_n\square P_n)\leq 12$, for sufficiently large $n$. 
There are studies showing constant (upper and lower, resp.) bounds on the nonrepetitive chromatic number for graph families, such as paths (3,3) \cite{thue1906uber}, cycles (3,4) \cite{currie2002there}, outerplanar graphs (7, 12) \cite{barat2007square,barat2008square,kundgen2008nonrepetitive} and planar graphs (11,768) \cite{dujmovic2013nonrepetitive,dujmovic2020planar}. 
As in these examples, our result in Theorem~\ref{thm_sk} shows that the $P_k$-chromatic number of the grid increases by a constant with respect to its chromatic number, independent of the value of $k$.
\begin{theorem}\label{thm_sk}
For any $k\geq 5$ and $m, n\geq k-2,$ $s_k(P_m\square P_n)=4$. 
\end{theorem}

In \cite{kirticsouglu2024coloring}, K\i rt\i \c{s}o\u{g}lu and the second author showed that 
$s_k(P_{k-3}\square P_n)=3$ 
for all $k\geq 5$ and $n\geq 1,$ by 
providing the following colorings for the case $k=5,6,$ 
which can be generalized to all $k\geq 5.$
\small
\begin{equation*}\label{eq:sk_G3m}
\begin{matrix}
1& 2 & 3& 1& 2 & 3&...\\
2 & 3 & 1& 2 & 3 & 1&...\\
\end{matrix}\qquad \qquad \qquad
\begin{matrix}
1 & 2 & 3 & 1 & 2 & 3 & \dots\\
2 & 3 & 1 & 2 & 3 & 1 & \dots\\
1 & 2 & 3 & 1 & 2 & 3 & \dots\\
\end{matrix}
\end{equation*} 
\normalsize
This coloring pattern having  
columns with alternating colors from $(1,2),$ $(2,3),$ $(3,1),$ respectively, yields a valid 3-coloring for any $k\geq 6$ and $n\geq 1$, showing $s_k(P_{k-3}\square P_n)= 3$. 
Note that in such colorings, a bicolored $P_k$ has to have at least $k-2$ vertices in the same column, thus cannot be found in $P_{k-3}\square P_n$ colored according to the pattern above. 
In \cite{kirticsouglu2024coloring}, it is also observed that for $k=5,6$, $s_k(P_{k-2}\square P_n)=4$ for all $n \geq k-2$ and this is conjectured to hold 
for all $k.$ By the monotonicity of $s_k(G)$ with respect to $k$, this result implies, for any $k\geq 5$ and $m, n\geq k-2,$ $s_k(P_m\square P_n)\leq 4$. With Theorem~\ref{thm_sk}, we confirm this conjecture showing that there is no proper 3-coloring of $P_m\square P_n$ avoiding a bicolored $P_k$, for $m,n\geq k-2$. In the following section, we provide the proof of Theorem~\ref{thm_sk}.

 \section{Main Result}

We call a maximal connected subgraph induced by vertices having only two colors a {\it bicolored component.} In particular, a bicolored component that has colors $c_1$ and $c_2$ is called a $c_1$-$c_2$ {\it colored component}. To prove Theorem~\ref{thm_sk}, we analyze bicolored components containing vertices from one of the sides of the grid. These bicolored components belong to one of the groups below: 
\begin{itemize}
    \item {\it peripheral bicolored component:} a bicolored component that has vertices in two opposite sides of the grid, i.e., top and bottom sides, or left and right sides.
    \item  {\it partial bicolored component:} a bicolored component that is not peripheral, but has vertices on at least one of the sides of the grid. 
\end{itemize}
In the rest, we assume that the sides of the grid that a partial bicolored component may intersect are top and left sides, since remaining cases are symmetric. We categorize each partial bicolored component $C$ as:
\begin{itemize}
    \item {\it Type 1:} if the vertices of $C$ are only on the top side,
    \item {\it Type 2:} if the vertices of $C$ are on the top and left sides.
\end{itemize}
We see examples of Type 1 and Type 2 partial bicolored components (as red-blue colored components) in Figure~\ref{fig:bdry} (a), and in Figure~\ref{fig:bdry} (b),(c), respectively. The following definitions are associated with a (partial or peripheral) bicolored component $C$.  
\begin{definition} When considering a bicolored component $C$ as a planar subgraph in the grid drawing, the {\it boundary of C} is a walk traversing the outer face of $C$, starting at any vertex, in clockwise direction.  
\end{definition}
For example, in Figure~\ref{fig:bdry} (c), the boundary of $C$ is the walk $(s, a, b, a, d, e, t, e, d, a, s).$ In Figure~\ref{fig:bdry} (b),  the boundary of $C$ is the walk $(s,a,b,c,d,c,t,c,e,a,s).$ In the rest, we make use of the term {\it outside of the grid} in the following sense. 
\begin{definition}\label{def:outside_grid}
 When walking clockwise on the boundary of $C$, if one's left hand stays outside the grid along some segment $S$ of that walk, we say $S$ is outside the grid. It is possible that a segment outside the grid is a single vertex.     
\end{definition}
For example, in Figure~\ref{fig:bdry} (c), the segment $(t, e, d, a, s)$ of the boundary  $(s, a, b, a, d, e, t, e, d, a, s)$ is outside the grid.  
\begin{definition}\label{defn:BC}
Given a fixed configuration of the grid, {\it a partial walk $B_C$} is a maximal subwalk of the boundary of a bicolored component $C$ without any segment outside of the grid. In addition, if $C$ is a partial bicolored component of Type-1 or Type-2, we let the  starting vertex of $B_C$ be the rightmost vertex of $C$ on the top side of the grid in order to make the partial walk $B_C$ of $C$ unique for such $C.$ 
\end{definition} 
By Definition~\ref{defn:BC}, on a fixed configuration of the grid, it is possible to have more than one partial walk on the boundary of $C$ only if $C$ is a peripheral bicolored component.
Some examples for the partial walks are shown in Figure~\ref{fig:bdry} (a), (b) and (c), where $B_C$ is described by the vertex sequence $(s, a, b, a, c, a, d, a, s)$, $(s, a, b, c, d, c, t)$, $(s, a, b, a, d, e, t)$, respectively. In Figure~\ref{fig:bdry} (a), the boundary of $C$ happens to be the same as $B_C.$ In Figure~\ref{fig:bdry} (d), the red-blue colored peripheral component containing $z_1$ has only one partial walk $(y_6, y_5, x_5, y_5, y_4, y_3, x_3, y_3, y_2, y_1, x_1).$ The green-blue colored peripheral component containing  $w_1$, has several partial walks, for example $(v_6, v_5, u_5)$ and $(w_1, w_2, x_2, w_2, w_3, w_4, x_4, w_4, w_5, v_5, v_6).$ The partial red-blue colored component containing $v_1$ has only one partial walk $(v_1, v_2, w_2, v_2, v_3, v_4, w_4,$ $ v_4, v_5, v_4, u_4).$  

 In the rest,  
 for simplicity, we write {\it 3-coloring} to refer to a proper 3-coloring of $P_{k-2}\square P_{k-2}$, $k\geq 5$, using colors red, blue and green, unless indicated otherwise. 
 With Lemma~\ref{lem-angle}, we make a generalization about the boundary structure of bicolored components.   
  To clarify the definition of the angles along $B_C$, we first present some examples. In Figure~\ref{fig:bdry} (a), the angles between the consecutive edges along $B_C$ are alternating between $90^{\circ}$ and $360^{\circ},$ whereas in Figure~\ref{fig:bdry} (c), these angles from $s$ to $t$ are $90^{\circ},$ $360^{\circ}$, $90^{\circ}$, $180^{\circ}$ and $90^{\circ}$. 
  In Figure~\ref{fig:bdry} (d), for the red-blue colored component with $z_1$ in it, these angles along the only partial walk (starting at $y_6$ ending at $x_1$) follow as $90^{\circ},$ $360^{\circ}$, $90^{\circ}$, $180^{\circ}$, $90^{\circ}$, $360^{\circ}$, $90^{\circ}, 180^{\circ}, 90^{\circ}.$ 
\begin{lemma}\label{lem-angle}
For any $k\geq 5$, let $C$ be a (partial or peripheral) bicolored component in a 3-coloring $c: V(P_{k-2}\square P_{k-2})\to \{c_1,c_2,c_3\},$ and label the vertices along a partial walk $B_C$ as $v_1, v_2, \dots, v_r$.  
Then, the following hold:
\begin{enumerate}
    \item $r\geq 3$. 
    \item The angles between the edges $v_1v_2$ and $v_2v_3$ and between $v_{r-2}v_{r-1}$ and $v_{r-1}v_r$ are $90^{\circ}$.  
    \item For $i\leq r-2$, the angle between the edges $v_iv_{i+1}$ and $v_{i+1}v_{i+2}$ along $B_C$ is $90^{\circ}$ if and only if $i$ is odd. 
    \item $r$ is an odd integer.
    \item 
    Let $c(v_1)=c_1$ and $C$ have colors $c_1$ and $c_2.$ Then, 
    there is a $c_1$-$c_3$ colored connected subgraph containing the vertices $v_i$ and $u_{(i+1)/2}$, for odd $i$, $1\leq i\leq r-2,$ where $u_{(i+1)/2}$ is the $c_3$-colored vertex on the $C_4$ containing $\{v_i, v_{i+1},v_{i+2}\}.$
\end{enumerate}
\end{lemma}

\begin{figure}[htbp]
    \centering
    \vspace{-.2cm}
    \includegraphics[scale=0.3]{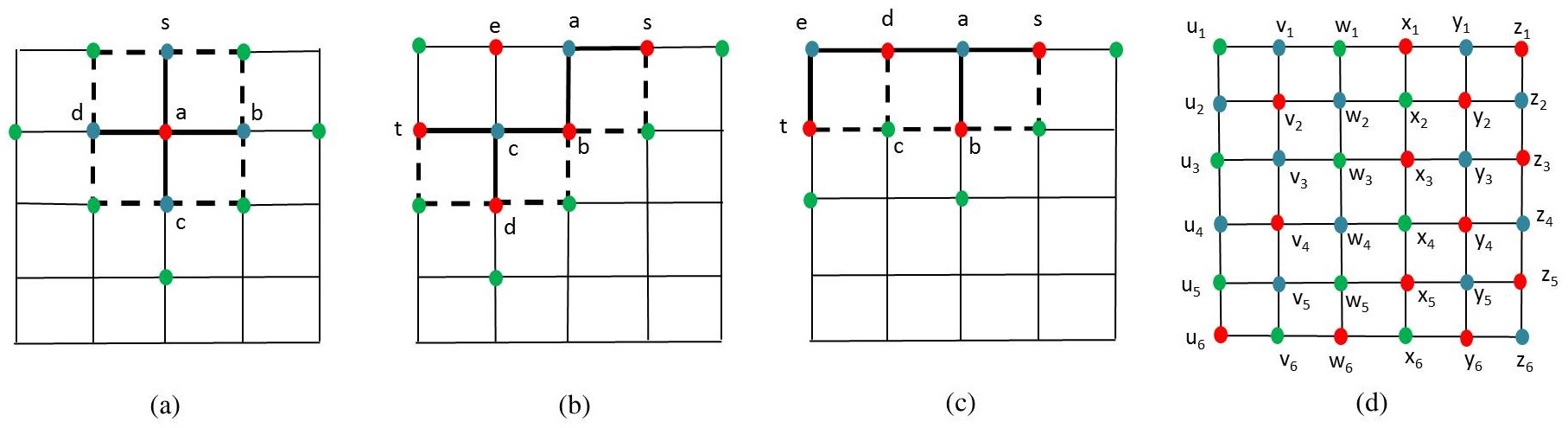}
    \vspace{-.2cm}
    \caption{In (a)-(c), examples of bicolored components show the edges of $B_C$ (bold edges) and the neighboring edges of another bicolored component (dashed edges). In (d), a coloring of the grid shows examples of peripheral bicolored components.}
    \label{fig:bdry}
    \vspace{-.2cm}
\end{figure}

\begin{proof}  
For simplicity, let us consider the 3-coloring described above, using colors red, blue and green and let $C$ be colored with red and blue.
Throughout the proof, we often make use of the fact that 
whenever a vertex on  the partial walk $B_C$ has a neighbor outside $C$, 
this neighbor has color green. 

\noindent
{\it (1)}  Since $C$ is bicolored, it has at least one edge. 
$C$ cannot be only a single edge, otherwise there are green colored adjacent vertices. 
For example, in Figure~\ref{fig:bdry} (d), if $y_1z_1$ is the only edge in its bicolored component, then $y_2$ and $z_2$ would be green. 
Thus, $C$ has at least two edges, the smallest case for $C$ having only the two edges incident to the top left corner vertex. Thus, $r\geq 3.$ 

\noindent
{\it (2)} Let the angle between $v_1v_2$ and $v_2v_3$ be $\theta$ and assume that $v_1v_2$ is a vertical edge, as in Figure~\ref{fig:bdry} (a). 

If $C$ is a partial bicolored component, recall that, $v_1$ is not part of the right side of the grid. Thus, in case $\theta$ has one of the values $180^{\circ}$,  $270^{\circ}$ or $360^{\circ}$, by definition of $B_C$, the neighbors of $v_1$ and $v_2$ to their right  are colored green and are adjacent, a contradiction. For example, in Figure~\ref{fig:bdry} (a), if $B_C$ were $(s,a,c,a,d,a,s)$, then $b$ had color green, a contradiction with proper coloring of the grid. 

If $C$ is a peripheral bicolored component, assume that $v_1$ is on the top side of the grid and $\theta\neq 90^{\circ}$. If $v_1$ is not on the top right corner of the grid, then same arguments above hold. Otherwise, $v_1v_2$ is a vertical edge on the right side of the grid. However, this is not possible, as by our assumption, $B_C$ is a clockwise walk and does not contain a segment outside the grid. 
Hence, $\theta=90^{\circ}.$

If  $v_1v_2$ is a horizontal edge as in Figure~\ref{fig:bdry} (b) and (c), similar arguments as above show that $\theta=90^{\circ}.$ 
The angle between the edges $v_{r-2}v_{r-1}$ and $v_{r-1}v_r$ cannot be other than $90^{\circ}$. 

\noindent
{\it (3)} Let $w, x, y, z$ represent the vertices $v_i, v_{i+1}, v_{i+2}, v_{i+3}$, respectively, for some $i.$ 
Call the angle between the edge pairs  $(wx, xy)$ and $(xy, yz)$, $\beta$ and $\gamma,$ respectively. In the following, we discuss invalid values for $(\beta, \gamma)$, letting $w$ have color red. We omit the symmetric cases, where the listed pairs of angles appear in reverse order. All cases are illustrated in Figure~\ref{fig:Case2b}. 

\begin{figure}[htbp]
\centering
\includegraphics[width=0.8\linewidth]{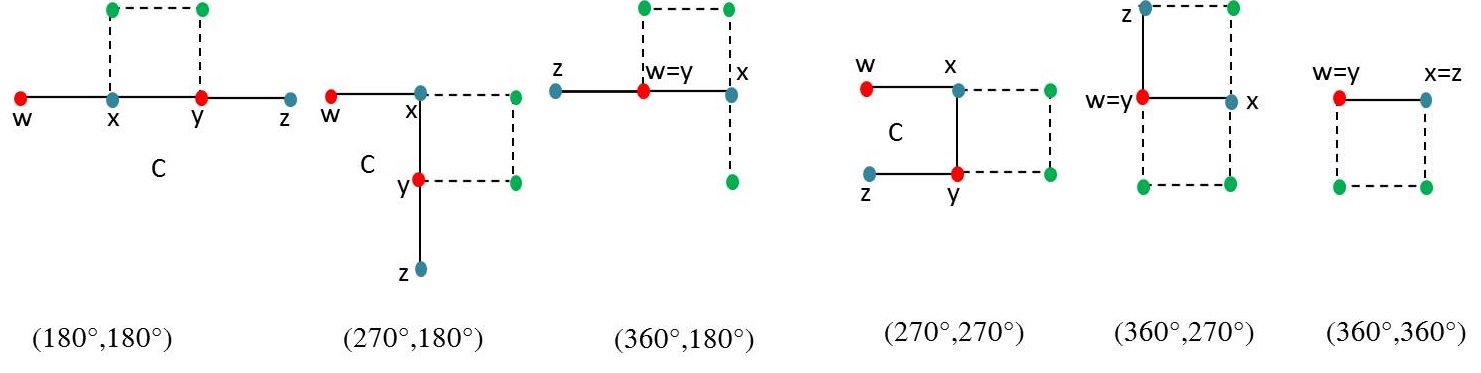}
  \vspace{-.2cm}
\caption{Invalid cases listing values of $(\beta,\gamma)$. 
}
\label{fig:Case2b}
\vspace{-.3cm}
\end{figure}

\begin{itemize}
    \item $(\beta, \gamma)=(180^{\circ},180^{\circ}), (270^{\circ},180^{\circ}), (270^{\circ},270^{\circ})$: As $B_C$ has no segment outside the grid, $x$ and $y$ have neighbors outside $C$ that are adjacent and green. 
     \item $(\beta, \gamma)=(360^{\circ},180^{\circ}),(360^{\circ},270^{\circ})$: In each of these cases, we use the fact that $x$ is not on any side of the grid, as in this ordering, $x$ cannot be $v_1$ or $v_r$ (endvertices of $B_C$). Thus, one of the $C_4$'s containing the edge $xy$, shown in Figure~\ref{fig:Case2b}, has two adjacent vertices that are not in $C$, hence green, a contradiction.  
    \item  $(\beta, \gamma)=(360^{\circ},360^{\circ})$: This case is not possible as it implies that $C$ consists of a single edge contradicting (1).
\end{itemize}
Thus, we observe that either $\beta$ or $\gamma$ is $90^{\circ}$ as listed in Figure~\ref{fig:Case2a}, where 
the vertices marked with * are implied to be in the bicolored component $C$ to have a proper coloring. In Figure~\ref{fig:Case2a} the case $(\beta, \gamma)=(90^{\circ},90^{\circ})$ gives a contradiction since the segment of $B_C$ should not be $(w,x,y,z)$ but $(w,z)$. Thus, (3) follows from (2).
\begin{figure}[htbp]
\centering
    \includegraphics[scale=0.3]{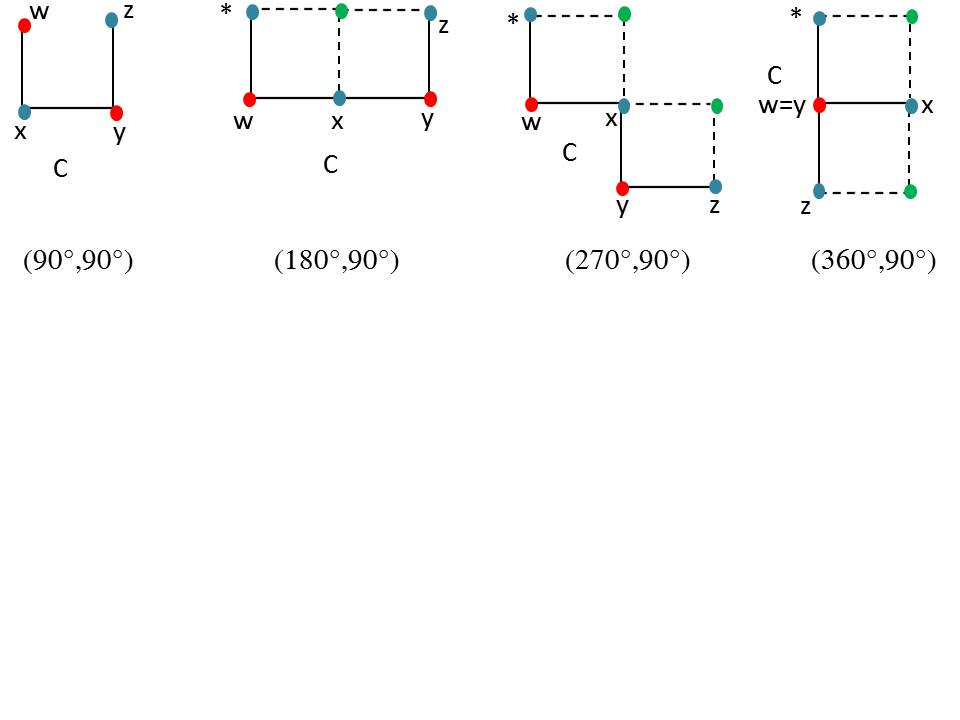} 
      \vspace{-3.6cm}
\caption{The only valid configurations for $(\beta,\gamma)$, except the leftmost case. The vertices marked with * are part of $C.$ (symmetric cases are omitted, where the listed pairs of angles appear in reverse order.)}
\label{fig:Case2a}
\end{figure}

\noindent
{\it (4)} This is implied by (2) and (3). 

\noindent 
{\it (5)} Let $u_{(i+1)/2}$ be the green vertex on the $C_4$ containing $\{v_i, v_{i+1},v_{i+2}\},$ for each odd $i$, $1\leq i\leq r-2$. Note that $u_{(i+1)/2}$ is not on $B_C$, since two consecutive degrees along $B_C$ cannot be both $90^{\circ}$. Also, vertices $u_{(i+1)/2}$ are not necessarily distinct, for example, in Figure~\ref{fig:bdry} (c), $u_2=u_3=c$. As given in the statement, let $c_1$ be the color of $v_1$ and let $C$ be $c_1$-$c_2$ colored. The edges $v_iu_{(i+1)/2}$ and $v_{i+2}u_{(i+1)/2}$ have colors $c_3$ and $c_1$, for each odd $i$, $1\leq i\leq r-2.$ 
So, we obtain a $c_1$-$c_3$ colored connected subgraph induced by the vertices $v_i$ and $u_{(i+1)/2}$, for odd $i$, $1\leq i\leq r-2.$ Some examples of this bicolored subgraph neighboring $C$ are shown in Figure~\ref{fig:bdry} with dashed edges. 
\end{proof}

Lemma~\ref{lem-case1} below is used to show  that Theorem~\ref{thm_sk} holds in case there is  a peripheral bicolored component in a 3-coloring of $P_{k-2}\square P_{k-2}$, $k\geq 5$. 

\begin{lemma}\label{lem-case1}
    For any $k\geq 5$, if a 3-coloring of $P_{k-2}\square P_{k-2}$ has a peripheral bicolored component, then there is a bicolored $P_k$. 
\end{lemma}
\begin{proof}
If the 3-coloring of $P_{k-2}\square P_{k-2}$ has a peripheral bicolored component, then let $P$ be a maximal bicolored path connecting two opposite sides, w.l.o.g. top and bottom sides, of the grid. 
We label the columns of the grid as $C_1,\dots,$ $C_{k-2}$ from left to right. Let $C_j$ be the leftmost column containing vertices from $P$. 
If they exist, we label the vertex set in $C_{j-1}, C_j, C_{j+1}, C_{j+2}$ as $v_s, w_s, x_s, y_s,$ $1\leq s\leq k-2$, respectively, $s$ indicating the row index with $s=1$ being the top row. Note that, it is possible that some of these columns do not exist depending on whether $P$ has edges on the right or left side of the grid. We present below a case analysis considering these possibilities as well and find a bicolored $P_k$ in each case. There can be at most one adjacent pair of vertices from different columns (jump across columns) in {\it P}, otherwise we have a bicolored $P_k.$ Thus, $P$ has vertices in at most two columns. Assume, w.l.o.g., that $P$ is a red-blue colored path and $w_1$ has color red in the rest of the proof.

\noindent
{\it Case 1.} $|V(P)|=k-2.$ 

The only possibility of this case is that $P$ contains only vertices from a single column (only $C_j$), i.e. $V(P)=\{w_1,w_2, \dots, w_{k-2}\}.$ Assuming that $C_{j+1}$ exists, $x_1$ has color green by the maximality of $P$. 
Since $w_2$ has color blue, $x_2$ has color red for a proper coloring. Also, the remaining vertices on $C_{j+1}$ have only red and green colors by maximality of $P$. Hence, the vertices on $C_{j+1}$ together with $w_1$ induce a red-green colored $P_{k-1}$. If $C_{j+1}$ does not exist and $P$ is the right side of the grid, then the same arguments yield a red-green colored $P_{k-1}$ similarly using the vertices on $C_{j-1}$ and $w_1.$ This contradicts the maximality of $P.$ 

\noindent
{\it Case 2.} $|V(P)|=k-1.$ 

By symmetry, we assume w.l.o.g. that the jump of $P$ (across columns) is from left to right, specifically from $C_j$ to $C_{j+1}$ with $j\geq 2$ and at the $i$'th row with $2\leq i\leq k-2.$ Thus, $P$ is $w_1, \dots, w_i, x_i, \dots, x_{k-2}.$ 
Note that $C_{j-1}$ exists since $j\geq 2$, and that $w_{i-1}$ and $x_{i-1}$ exist since $i\geq 2.$ By arguments as in Case 1, $v_1$ is (colored) green, $v_2$ is red, $v_3$ is green,..., $v_i$ is red or green (depending on the parity of $i$). The following can be proved similarly. 

\begin{claim}\label{rem:nbrs}
$v_1,\dots,v_i$ and $x_1,\dots,x_{i-1}$ are green-red colored paths with $v_1$ and $x_1$ colored green. If $\alpha$ is the color of $x_{k-2}$, then $w_{k-2},\dots,w_{i+1}$ and $y_{k-2},\dots,y_i$ (if it exists) are green-$\alpha$ colored paths with $w_{k-2}$ and $y_{k-2}$ colored green. 
\end{claim}

If $k$ is even and $i=k-2$, then $v_1,w_1,x_1,\dots,x_{k-2}$ is a green-red colored path by Claim~\ref{rem:nbrs}, contradicting the maximality of $P.$ If $k$ is even, $i<k-2$ and $i$ is even, then $w_{i+1}$ is colored red, which means $w_1,\dots,w_{i+1},x_{i+1},\dots,x_{k-2}$ is also a red-blue colored path. Thus, if $k$ is even and $i<k-2$, we may assume that $i$ is odd. If $k$ is even, $i$ is odd, and $i<k-2$, then $w_1,v_1,\dots,v_i,w_i,\dots,w_{k-2}$ is a red-green colored path by Claim~\ref{rem:nbrs}, contradicting the maximality of $P.$ Thus, we have the following. 
\begin{claim}\label{clm:k_odd}
$k$ is odd.
\end{claim}
If $i$ is odd (i.e., $k$ and $i$ are both odd), then by Claim~\ref{rem:nbrs}, $x_{i-1}$ is colored red and $w_{i+1}$ (if exists) is colored blue, and therefore the path $w_1\dots,w_{i-1},x_{i-1},x_i,w_i,w_{i+1},x_{i+1},\dots,x_{k-2}$ (or the part of it that exists) is a red-blue colored path, contradicting the maximality of $P.$ Thus, we have the following. 
\begin{claim}\label{clm:i_even}
    $i$ is even (which also means that $i<k-2$).
\end{claim}
To complete the proof, it suffices to produce a contradiction assuming that $k$ is odd and $i$ is even. 
If $j\leq k-4$, then $C_{j+2}$ exists and by Claim~\ref{rem:nbrs}, the path $x_i, \dots, x_1, w_1$, $v_1,\dots,v_i$ is a red-green colored $P_{2i+1}$ and 
the path $w_i, \dots, w_{k-2}, x_{k-2}, y_i, \dots, y_{k-2}$ 
is a blue-green colored $P_{2(k-i-1)+1}$. One of these bicolored paths has at least $k$ vertices, depending on the value of $i$.

If $j=k-3$, then $C_{j+2}$ does not exist. In this case, it is not guaranteed to have a bicolored $P_k$ only using vertices in $P$ and their neighbors.  So, we use Lemma~\ref{lem-angle} to show the existence of a bicolored $P_k$.  Let $C$ be the bicolored red-blue component containing $P.$ If $C$ contains any vertex from the top or bottom sides outside $P,$ then $C$ contains a red-blue $P_k.$ If $C$ contains any vertex from the left side, then the red-blue path in $C$ connecting $P$ to the left side together with $P$ yield a bicolored path with at least $k$ vertices. Thus, we assume that $C$ is not connected to any side of the grid (other than the right side) outside $P$. This implies that there exists a partial walk of $C$ from $x_{k-2}$ to $w_1$. However, (4) of Lemma~\ref{lem-angle} yields that $x_{k-2}$ and $w_1$ have the same color, a contradiction. 
\end{proof}

\subsection*{ Proof of Theorem~\ref{thm_sk}}
We know that $s_k(P_m\square P_n)\leq 4$ as $s_5(P_m\square P_n)=4$, for any $m,n\geq 3$, $k\geq 5$ \cite{kirticsouglu2024coloring}. 
To prove that $s_k(P_m\square P_n)\geq 4$, 
it suffices to show that $s_k(P_{k-2}\square P_{k-2})\geq 4.$ 
We consider any 3-coloring of $P_{k-2}\square P_{k-2}$ for $k\geq 5$ and show using Claim~\ref{clm-fin} that it has a peripheral bicolored component. Hence, there is a bicolored $P_k$ by Lemma~\ref{lem-case1}. 

Let $P_{k-2}\square P_{k-2}$ be colored with colors $\{c_1,c_2,c_3\}$ and 
assume that such a 3-coloring does not have a peripheral bicolored component. 
Let $C$ be a partial, $c_1$-$c_2$ colored component containing vertices from the top side and possibly left side of the grid. 
Let $B_C= (v_1, v_2, \dots, v_r)$ be the partial walk of $C$, where $v_1$ has color $c_1$.  
We need the following definitions associated with $C$:
\begin{itemize}
    \item $D_C$: By Lemma~\ref{lem-angle}, there is a $c_1$-$c_3$ colored connected subgraph, 
    induced by $v_i$ and $u_{(i+1)/2}$, for odd $i$, $1\leq i\leq r-2,$ where $u_{(i+1)/2}\not \in C$ is the $c_3$-colored vertex on the $C_4$ containing $\{v_i, v_{i+1},v_{i+2}\}.$ Let $D_C$ be the bicolored component containing this $c_1$-$c_3$ colored connected subgraph.
    \item  $\mathcal{W}_C$: Let this be the collection of $C_4$'s induced by $\{u_{(i+1)/2}, v_i, v_{i+1},v_{i+2}\}$, for odd $i$, $1\leq i\leq r-2$. As $r\geq 3$, $\mathcal{W}_C\neq \emptyset$. 
    \item $R_C$: This is the region of the grid enclosed by $B_C$ and the walk along the four sides of the grid in clockwise direction starting at $v_1$ ending at $v_r$. In other words, it is the region to the left when walking on $B_C$ from $v_1$ to $v_r$ that is enclosed by the four sides of the grid. 
\end{itemize}

\begin{claim}\label{clm-fin}
    When $C$ is replaced with $D_C$ iteratively, $D_C$ will be a peripheral bicolored component after a finite number of iterations.  
\end{claim}
\begin{proof}
By definition and by Lemma~\ref{lem-angle}, at each iteration, $D_C$ contains the vertices $v_1$ and $v_r$ of $B_C=(v_1,\dots,v_r)$, where $v_1$ is on the top side and $v_r$ is either on the top side or left side of the grid. Let $B_C'=(v_1',\dots,v_r')$ be the partial walk of $D_C$. Assuming $D_C$ is not peripheral, possible scenarios at each iteration are:
\begin{enumerate}
    \item $C$ and $D_C$ have some vertices only from the top side and no other side of the grid. Thus, $v_1$, $v_1'$, $v_r$, $v_r'$ are on the top side of the grid. 
    \item $C$ has some vertices only from the top side and no other side of the grid, whereas $D_C$ has vertices from the top and one of the right or left sides of the grid. In case $D_C$ has some vertex on the right side, we flip the grid to exchange the right and left sides of the grid so that  $D_C$ satisfies our assumptions. Thus, $v_1$, $v_1'$ are on the top side, $v_r$ is on the top side, $v_r'$ is on the left side.  
    \item Both $C$ and $D_C$ have some vertices from the top and left side of the grid. Thus, $v_1$, $v_1'$ are  on the top side of the grid, $v_r$, and $v_r'$ are on the left side of the grid.
\end{enumerate} 

\begin{figure}[htbp]
    \centering
    \vspace{-0.4cm}
    \includegraphics[width=0.5\linewidth]{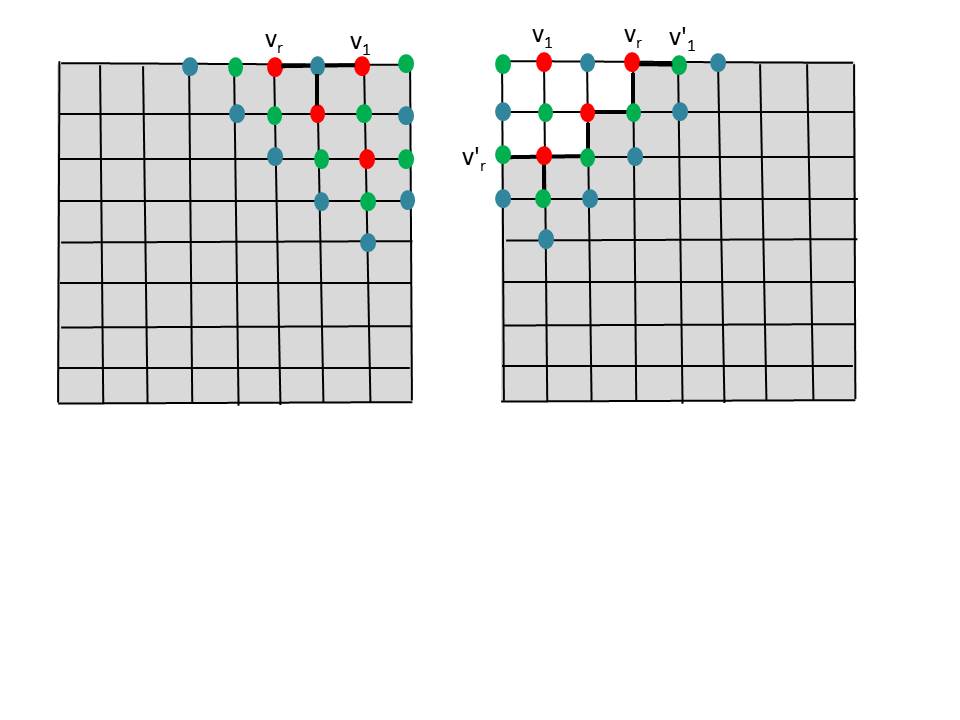}
    \vspace{-2.6cm}
    \caption{An example of Scenario (2), with the left and right grids showing the current $C$ before and after the iteration (and the flip), respectively. $\mathcal{R}_C$ is shown with the shaded region and $B_C$ is marked with bold edges. }
    \label{fig:main}
\end{figure}
In Scenario except (2), the grid is not replaced with any of its automorphisms and 
the sides of the grid to which the start and end vertices of $B_C$ belong do not change at any iteration. 
Note that an iteration in Scenario (2) may occur at most once. In all iterations, except possibly one with Scenario (2), we observe that $v_1'$ is to the right of $v_1$ on the top side of the grid or $v_1=v_1'$. Similarly, in all iterations in Scenario (1), $v_r'$ is to the left of $v_r$ on the top side of the grid or $v_r=v_r'$, and in Scenario (3), $v_r'$ is below $v_r$ on the left side of the grid or $v_r=v_r'$.   

Note that, after an iteration, the squares enclosed by the $C_4$'s in $\mathcal{W}_C$ are missing in $R_C$. As $\mathcal{W}_C\neq \emptyset$, 
$R_C$ becomes smaller after each iteration with Scenario (1)-(3). Thus, 
as long as $C$ is not peripheral, $v_1=v_1'$ or $v_r=v_r'$ can happen at most once for each vertex, as there are exactly two $C_4$'s that contain $v_1$ (same for $v_r$) and no $C_4$ belongs to $R_C$ in more than one iteration. In that sense, each vertex plays the role of $v_1$ or $v_r$ with multiplicity at most two throughout all iterations. Consequently, at each consecutive pair of iterations, the position of $v_1$ approaches the top right corner by at least one edge, and the position of $v_r$ approaches the top left or the bottom left corner of the grid  by at least one edge, depending on the case. These observations show that after some finite number of iterations, none of the Scenarios (1)-(3) will hold and $D_C$ will become a peripheral bicolored component.    
\end{proof}

\acknowledgements
\label{sec:ack}
The authors would like to thank the referees for
their valuable comments that improved the quality of this paper.

%\nocite{*}
\bibliographystyle{abbrvnat}
% use the following instead if you encounter problems 
%\bibliographystyle{alpha}
\bibliography{arxiv_submission2025}
\label{sec:biblio}

\end{document}